\documentclass[10pt,reqno]{amsart}
  \usepackage{geometry}
\usepackage{amssymb,upref,enumerate}
\usepackage{tikz}
\usepackage{mathabx}
\usepackage{amsmath,amssymb,yfonts,amscd,amsthm,amsxtra,mathrsfs}
\usepackage{latexsym}
\usepackage{soul}
\usepackage[T1]{fontenc}
\usepackage{tgtermes}
\usepackage[latin9]{inputenc}

\usepackage{cite} 

\usepackage[
colorlinks=true,
linkcolor=blue,
citecolor=blue,
urlcolor=blue,
pagebackref,
]{hyperref}

\makeatletter
\def\Ddots{\mathinner{\mkern1mu\raise\p@
\vbox{\kern7\p@\hbox{.}}\mkern2mu
\raise4\p@\hbox{.}\mkern2mu\raise7\p@\hbox{.}\mkern1mu}}
\makeatother

\def\Xint#1{\mathchoice
{\XXint\displaystyle\textstyle{#1}}%
{\XXint\textstyle\scriptstyle{#1}}%
{\XXint\scriptstyle\scriptscriptstyle{#1}}%
{\XXint\scriptscriptstyle\scriptscriptstyle{#1}}%
\!\int}
\def\XXint#1#2#3{{\setbox0=\hbox{$#1{#2#3}{\int}$}
\vcenter{\hbox{$#2#3$}}\kern-.5\wd0}}

\def\dashint{\Xint-}

\def\avgint{\Xint-}

\newcommand{\vertiii}[1]{{\left\vert\kern-0.25ex\left\vert\kern-0.25ex\left\vert #1 
    \right\vert\kern-0.25ex\right\vert\kern-0.25ex\right\vert}}

\newtheorem{theorem}{Theorem}[section]

  \theoremstyle{plain}
  
  \theoremstyle{plain}

\newtheorem{corollary}[theorem]{Corollary}

\theoremstyle{definition}

\newtheorem{lemma}[theorem]{Lemma}

\newtheorem{remark}[theorem]{Remark}

\newcommand{\strt}[1]{\rule{0pt}{#1}} 

\newcommand{\X}{\mathcal{S}}
\newcommand{\dif}{\mathrm{d}} 

\def\al{{\alpha}}
\def\R{\mathbb R}
\def\N{\mathbb N}

\def\Z{\mathbb Z}
\def\Sn{\mathbb S}

\def\ra{\rightarrow}

\def\bey{\begin{eqnarray*}}
\def\eey{\end{eqnarray*}}

\allowdisplaybreaks

\newcommand{\cpm}{\color{magenta}} 

\newcommand{\ed}{\color{black}} 

\begin{document}

\title[A new look at the subrepresentation formulas]{A new look at the subrepresentation formulas}

\author{Cong Hoang}

\author{Kabe Moen}

\author{Carlos Perez}

\address{Carlos P\'erez \\
 Department of Mathematics \\
 University of the Basque Country \\
 Ikerbasque and BCAM \\
 Bilbao, Spain\\ \textsf{cperez@bcamath.org}}

\address{Kabe Moen \\
 Department of Mathematics \\
 University of Alabama \\
 Tuscaloosa \\
 Alabama, USA, \\ \\ \textsf{kabe.moen@ua.edu}}

\address{Cong Hoang \\
 Department of Mathematics \\
 Florida A\&M University \\
 Tallahassee \\
 Florida, USA \\ \textsf{cong.hoang@famu.edu}}

\thanks{ 
C. P.  is supported by grant  PID2020-113156GB-I00, Spanish Government; by the Basque Government through grant 
IT1615-22 and the BERC 2014-2017 program and by BCAM Severo Ochoa accreditation SEV-2013-0323, Spanish Government.  K.M. is supported by the Simons Foundation grant number 160427. }


\vspace{-1cm}

\begin{abstract} 

We extend the subrepresentation formula 
\begin{equation*}
|f(x)|\le c_n\,I_1(|\nabla f|)(x)
\end{equation*}
in several ways.  First, we consider more general $A_1$-potential operators on the right-hand side and prove local and global pointwise inequalities for these operators.  Second, we show that we can improve the right-hand side using fractional derivatives.  Finally, we extend our results to rough singular integral operators, similar to the main result in \cite{HMP1}.  

\end{abstract}

\maketitle


%

\section{Motivation}

 It is well known that  the pointwise inequality in $\R^n$, $n\geq2$,
\begin{equation}\label{ptwiseglobal}
|f(x)|\le c_n\,I_1(|\nabla f|)(x),
\end{equation}
plays a central role in relating a function and its smoothness by $L^p$ type estimate. Recall here that 
$I_1$ is the Riesz potential operator of order $\al=1$ given by 
$$I_\al f(x)=\int_{\R^n}\frac{f(y)}{|x-y|^{n-\al}}\,dy, \quad 0<\al<n.
$$
This estimate goes back to Sobolev and gives rise to the Gagliardo-Nirenberg-Sobolev inequality 
$$\|f\|_{L^{p^*}(\R^n)}\leq C\,\|\nabla f\|_{L^p(\R^n)},$$
when $f\in C_c^1(\R^n),$ $1\leq p<n$, and $\frac1{p^*}=\frac1p-\frac1n.$
Moreover, local versions of \eqref{ptwiseglobal}, like, 
$$|f(x)-f_B|\leq C\,I_1(\mathbf 1_B|\nabla f|)(x), \qquad x\in B,$$
are paramount for studying the Poincar\'e-Sobolev inequalities
$$\left(\dashint_B |f(x)-f_B|^q\,dx\right)^{\frac1q}\leq C\,r(B)\left(\dashint_B|\nabla f|^p\,dx\right)^{\frac1p},$$ 
 where $B$  could be  euclidean ball in $\R^n$, or other type of domains.  
With this result as a model, many similar results can be found where the gradient operator $\nabla$ is substituted with generalized gradients, and the Riesz potential operator is replaced by other types of potential operators. 
More precisely, one can bound the norms of $f$ by the norms of $Xf$ provided there is an integral operator $T$ which is bounded on appropriate weighted $L^p$ spaces and for which the pointwise estimate 
$$|f(x)|\leq c\,T(|Xf(x)|)$$
is valid. Such inequality is called a subrepresentation.  For example, this pointwise estimate is known
to hold for vector fields of H\"ormander type when $T$ is given by
\[
Tf(x)=\int_{\R^n} \frac{d(x,y)}{|B(x,d(x,y))|}\,f(y)\,dy,
\]
where $d(x,y)$ is the associated Carnot--Carath\'eodory metric, $B(x,r)$ denotes the metric ball with center $x$ and radius $r$, and $|B(x,r)|$ is the Lebesgue measure of the ball.

More generally, there are "potential type operators" within the context of doubling metric measure spaces $(\X,d,\mu)$  of the form 
\begin{equation}
Tf(x)=  \int_{\X} K(x,y)\,f(y)\, d\mu(y),
\label{potential}
\end{equation}
where the kernel $K(x,y)$ is non-negative and satisfying appropriate estimates.  
An important class of examples with metrics other than the usual
Euclidean metric consists of potential operators related to the
regularity of subelliptic differential equations. In particular,
vector fields of H\"ormander type (see \cite{H}) as well as the
classes of non-smooth vector fields studied in \cite{FL} lead to
integral operators of the type we will study. In addition, the
differential operators of Grushin type considered in \cite{FGuW}
(at least in the simplest case of Lebesgue measure) are related to
integrals of type (\ref{potential}). In fact, for all these
examples, the associated potential operator has the form
\begin{equation}
Tf(x)= \int_{\X}\frac{d(x,y)}{\mu(B(x,d(x,y)))}\, f(y)\,d\mu(y),
\label{subelliptic}
\end{equation}
where $d(x,y)$ is a distance function that is naturally related to the vector fields. For a nice survey on these operators we refer the reader to the work of Kairema \cite{AK}.

Motivated by these results, we propose improvements of \eqref{ptwiseglobal} by considering measures $\mu$ induced by a non-negative locally Lebesgue integrable function $w$. More precisely, we will consider potential type operators $T_w$ defined by %
\begin{equation}\label{PotA1}
T_w f(x) :=\int_{\R^n} \frac{|x-y|}{w(B(x,|x-y|))}  f(y)\,w(y)\,dy,
\end{equation}
where $w(B)=\int_Bw(x)\,dx$. We shall also consider the fractional version of $T_w$ defined by
$$T_{w,\al} f(x) :=\int_{\R^n} \frac{|x-y|^\al}{w(B(x,|x-y|))}  f(y)\,w(y)\,dy,$$
for any $0<\alpha<n$.

 Recall that a weight $w$ is said to be in the class $A_1$ if and only if for any ball $B$,
$$\dashint_B w(x)\,dx \le C\, w(x),$$
for almost every $x\in B$. The least constant $C$ for the above inequality is denoted as $[w]_{A_1}$. 
With weights from $A_1$, we will be able to derive a more general version for \eqref{ptwiseglobal}, namely
$$|f(x)|\leq 
c_{n,w}\, T_w(|\nabla f|)(x).$$
%
%

In addition, it was recently shown by the authors in \cite{HMP1,HMP2} that \eqref{ptwiseglobal} has a powerful extension to operators.  Consider a degree zero homogeneous function $\Omega \in L^1(\mathbb S^{n-1})$. The rough singular integral operator is defined by
$$T_\Omega f(x)=\mathsf{p.v.}\int_{\R^n}\frac{\Omega(y)}{|y|^n}f(x-y)\,dy,$$
and its associated maximal truncation operator is given by
$$T^\star_{\Omega} f(x)=\sup_{t>0}\left|\int_{|y|>t}\frac{\Omega(y)}{|y|^{n}}f(x-y)\,dy\right|.$$
It is well-known that the size of $\Omega$ determines the boundedness properties of the operators $T_\Omega$ and $T^\star_\Omega$.  It was shown in \cite{HMP1} that 
\begin{equation}\label{Tstarbound} T^\star_{\Omega} f(x)\leq 
c_{n,\Omega}\, I_1(|\nabla f|)(x)\end{equation}
when $\Omega\in L^{n,\infty}(\mathbb S^{n-1})$.  We will extend this estimate to $T_w$, namely,
$$T^\star_\Omega f(x)\leq c_{n,w,\Omega}\, T_w(|\nabla f|)(x),$$
for an appropriate class of weights $w$. 

Another approach to improve inequality \eqref{ptwiseglobal} involves exploring smaller operators that could replace $I_1(|\nabla \cdot|)$ on the right-hand side. One promising candidate is $I_\alpha(\mathcal{D}^\alpha)$ where $\mathcal{D}^\alpha$ is a nonlinear fractional differential operator. The fractional derivative operator, $|\nabla|^\alpha=(-\Delta)^{\alpha/2}$, is defined using the Fourier transform for $f\in C_c^\infty(\R^n)$
$$\mathcal{F}(|\nabla|^\al f)(\xi)=|\xi|^{\al}\hat{f}(\xi).$$
For $0<\alpha<2$, the fractional derivative can be realized via the integral formula
$$|\nabla |^\alpha f(x)=c_{n,\alpha}\,\mathsf{p.v.}\int_{\R^n}\frac{f(x)-f(y)}{|x-y|^{n+\alpha}}\,dy$$
(we refer to the book by Ponce \cite[p. 246]{Pon} for a related variant).
When $0 < \alpha < 1$, the above integral converges absolutely, and hence we may define the nonlinear fractional differential operator introduced in \cite{Sp},
$$\mathcal D^{\al}f(x)=\int_{\R^n}\frac{|f(x)-f(y)|}{|x-y|^{n+\al}}\,dy.$$

Indeed, we shall prove that 
\begin{equation}\label{improvedbound}
    T^\star_{\Omega} f(x)\leq C_{n,\al,\Omega}\, 
(1-\alpha)\,I_{\al}(\mathcal D^{\al}f)(x), \qquad f\in C^\infty_c(\R^n), \, 0<\al<1
\end{equation}
which improves the main result in \cite{HMP1} as we shall prove later that 
\begin{align*}
 (1-\alpha)\, I_{\al}(\mathcal D^{\al}f)(x) \leq c_{n,\al}\, I_1(|\nabla f|)(x),
\end{align*}
where the constant $c_{n,\al}$ is bounded by a dimensional constant.

\hfill

\section{The theorems}


When $w\in A_1$, we have an analog of \eqref{ptwiseglobal}, as stated in the theorem below.

\begin{theorem} \label{identity}  Let $w\in A_1$. There exists a dimensional constant $c_n$ such that for any Lipschitz function $f$
that has compact support, we have the global estimate
\begin{equation}\label{ptwiseglobalA1} |f(x)|\leq 
c_n\,[w]_{A_1}\, T_w(|\nabla f|)(x).
\end{equation}
\end{theorem}

\begin{remark} 
 
We remark that \eqref{ptwiseglobalA1} has been obtained using a variant of the usual estimates, where some "holes" have been created to obtain a family of disjoint annuli, as seen in \eqref{Estimatewithwholes} from the proof.  
This family of "holes" allows us to insert the metric length $|x-y|$ inside the integral. It seems that the usual methods do not permit this possibility. 
A similar remark holds for the rest of the theorems stated in the section.
\end{remark}

One can replace the left-hand-side of \eqref{ptwiseglobalA1} by $T^\star_\Omega(f)$, and this yields an improvement for the main result in \cite{HMP1}. 

\begin{theorem} \label{rough} Suppose that $\Omega$ is homogeneous of degree zero with zero average on $\mathbb S^{n-1}$ and belongs to $L^{n,\infty}(\mathbb S^{n-1}).$  Let $w\in A_1$. For any smooth function $f$ with compact support, 
\begin{equation}\label{ptwiseglobalA1-roughSIO} T^\star_{\Omega} f(x)\leq 
c_n \|\Omega\|_{\strt{2ex}L^{n,\infty}(\mathbb S^{n-1})} [w]_{A_1} \, T_w(|\nabla f|)(x)
\end{equation}
where $T_w$ is the potential type operator defined above. 
\end{theorem}

\begin{remark} 
 We have not encountered either \eqref{ptwiseglobalA1} or \eqref{ptwiseglobalA1-roughSIO} in the literature with $w$ belonging to $A_1$ (or any $A_p$), which may lead to new $L^p$--$L^q$ estimates.
\end{remark}


\begin{theorem} \label{fractional} Let $\alpha\in (0,1)$. Suppose that $\Omega$ is homogeneous of degree zero with zero average on $\mathbb S^{n-1}$ and belongs to $L^{\frac{n}{\alpha},\infty}(\mathbb S^{n-1})$. Then,
\begin{equation}\label{ptwisebdd} T^\star_{\Omega} f(x)\leq c_n(1-\alpha)\|\Omega\|_{L^{\frac{n}{\alpha},\infty}(\mathbb S^{n-1})}\,I_{\alpha}(\mathcal D^{\alpha}f)(x), \qquad f\in C^\infty_c(\R^n).
\end{equation}
\end{theorem}

To see that Theorem \ref{fractional} improves the main result in \cite{HMP1}, we have the following lemma.

\begin{lemma}\label{I1-domination}
Let $0<\alpha<1$ and $f\in C^\infty_c(\R^n)$. Then,
\begin{align*}
    (1-\alpha)\,I_{\alpha}(\mathcal D^{\al}f)(x) \leq c_{\alpha,n} \, I_1(|\nabla f|)(x)
\end{align*}
where $c_{\alpha,n}$ is bounded by a dimensional constant as $\alpha\ra 1^-$.
\end{lemma}

\begin{remark}
  The explicit constant is given by
  $$c_{\alpha,n}=\frac{(1-\alpha)\,\pi^{\frac{n-1}{2}}\,\Gamma(\frac{1-\alpha}{2})\,\Gamma(\frac{\alpha}{2})\,\Gamma(\frac{n-1}{2})}{\alpha\,\Gamma(\frac{n+\alpha-1}{2})\,\Gamma(\frac{n-\alpha}{2})}.$$
  Since
  $$(1-\alpha)\,\Gamma(\frac{1-\alpha}{2})=2\,\Gamma(\frac{3-\alpha}{2})$$
  we have the limit  
 $$\lim_{\alpha\ra 1^-}c_{\alpha,n}=\frac{2\pi^{n/2}}{\Gamma(\frac{n}{2})}=\sigma(\mathbb S^{n-1})
 $$
 where $\sigma$ is the surface measure of the unit sphere.
%
%
%
\end{remark}

To prove Theorem \ref{fractional}, we shall need the following Poincar\'e-Sobolev type estimates related to the localized version of the nonlinear fractional differential operator $\mathcal D^\al$; namely, for any $0<\al<1$ and any cube (or ball) $Q$,
\begin{equation}\label{PS-BBM}
\left(\avgint_Q |f(x)-f_Q|^{\left(\frac{n}{\al}\right)'}\, dx\right )^\frac{1}{\left(\frac{n}{\al}\right)'} \le c_{n}\, (1-\al) \,\ell(Q)^{\al} \avgint_{Q} \int_Q \frac{\lvert f(x)-f(y)\rvert}{\lvert x-y\rvert^{n+\al}}\,dy\,dx
\end{equation}
which can be found in \cite{BBM2}. A different approach and improvements can be found in \cite{HMPV} and \cite{MPW} where everything is derived from the following $(1,1)$ case:
\begin{align}\label{PS-BBMspecial}
    \avgint_Q |f(x)- f_{Q}|\,dx\leq  c_n\,(1-\al) \,\ell(Q)^{\al} \avgint_Q \int_Q\frac{|f(x)-f(y)|}{|x-y|^{n+\al  }}\,dy\,dx.
\end{align}
It is interesting to note that whenever $\al\in(0,1)$, 
$$
(1-\al)\,\ell(Q)^\al \avgint_Q \int_Q\frac{|f(x)-f(y)|}{|x-y|^{n+\al  }}\,dy\,dx\leq c_n\ell(Q)\,\avgint_Q |\nabla f(x)| \,dx.
$$
Further, they showed in \cite{BBM1} that the left-hand side converges to the right-hand side as $\al \rightarrow 1^-$.

For any $0<p,q<\infty$ and a measure space $(X,\mu)$, the Lorentz space $L^{p,q}(\mu)$ is the collection of functions that satisfy
$$\|f\|_{L^{p,q}(\mu)}=\left(p\int_0^\infty t^q\mu(\{x\in X:|f(x)|>t\})^{\frac{q}{p}}\frac{dt}{t}\right)^{\frac1q}<\infty.$$
When $q=\infty$, $L^{p,\infty}(\mu)$ is the weak $L^p$ space defined by
$$\|f\|_{L^{p,\infty}(\mu)}=\sup_{t>0}t\mu(\{x\in X:|f(x)|>t\})^{\frac1p}<\infty.$$
The normalized Lorentz average associated to a cube $Q$ (or a ball) is defined to be 
$$\|f\|_{L^{p,q}(Q)}:=\|f\|_{L^{p,q}(Q,\frac{dx}{|Q|})}=\frac{1}{|Q|^{\frac1p}}\|f\mathbf 1_Q\|_{L^{p,q}(\R^n)}.$$
There is a better version of \eqref{PS-BBM}, stated for the Lorentz average:
\begin{equation}\label{PS-BBM-Lorentz}
\|f-f_Q\|_{L^{\left(\frac{n}{\al}\right)',1}(Q)} \le c_{n}\, (1-\al) \,\ell(Q)^{\al} \avgint_{Q} \int_Q \frac{\lvert f(x)-f(y)\rvert}{\lvert x-y\rvert^{n+\al}}\,dy\,dx,
\end{equation}
which was proved in \cite{MPW}.

\hfill

Motivated by Theorem \ref{fractional}, we have the following fractional versions of Theorems \ref{identity} and \ref{rough}.

\begin{theorem} \label{identityfractional}  Let $w\in A_1$ and $0<\al<1$. 
There exists a dimensional constant $c_n$ such that for any Lipschitz function $f$
that has compact support, 
we have
\begin{equation} \label{globalfrac}
|f(x)|\leq 
 c_n (1-\al)\, [w]_{A_1}\, T_{w,\al}(\mathcal{D}^\al f)(x).
\end{equation}
\end{theorem}

\begin{remark}
We have yet to encounter this estimate with $w=1$ in the literature or in recent books like \cite{EE} or \cite{L}. 
 There are not any estimates like these even in the simpler case when the right-hand side does not have the so-called Bourgain-Brezis-Mironescu factor $1 - \alpha$.
\end{remark}

 Similarly, we have the following result replacing the identity operator in the left-hand side of \eqref{globalfrac} 
by $T^\star_\Omega(f)$ which yields another improvement for some of the main results in \cite{HMP1}.

\begin{theorem} \label{roughfractional} Suppose that $\Omega$ is homogeneous of degree zero with zero average on $\mathbb S^{n-1}$ and belongs to $L^{\frac{n}{\al},\infty}(\mathbb S^{n-1}).$   Let $w\in A_1$ and $0<\al<1$.  Then  
there exists a dimensional constant $c_n$ such that \ed for any for any smooth function $f$ with compact support, 
\begin{equation} T^\star_{\Omega} f(x)\leq 
 c_n(1-\al)\, \|\Omega\|_{\strt{2ex}L^{\frac{n}{\al},\infty}(\mathbb S^{n-1})} [w]_{A_1} \, T_{w,\al}(\mathcal{D}^\al f)(x).
\end{equation}
\end{theorem}

\hfill

\section{Proofs of the theorems}

\begin{proof}[Proof of Theorem \ref{identity}]

Recall the $(1,1)$ Poincar\'e inequality, for a dimensional constant $c_n$
$$ 
\avgint_{B}\lvert f(x)-f_{B}\rvert\,dx \leq  c_n\, r(B)\,\avgint_B |\nabla f(x)| \,dx
$$
for every ball $B$.

Fix a point $x\in \R^n$ and let $R>0$ denote the radius of the ball $B$ centered at $x$.  Let  $\{B_k\}_{k\in\N}$ be the family of nested balls centered at $x$ such that $B_1=B$ and $r(B_k)=2^{1-k}R$. 
Then, 
\begin{equation*}
\lvert f(x)-f_{B}\rvert=\left \lvert\lim_{k\to\infty} f_{B_k} - f_{B_1}\right\rvert\le \sum_{k\in\N} \left\lvert f_{B_{k+1}}- f_{B_k}\right\rvert.
\end{equation*}
Using the dyadic structure of the chain and the Poincar\'e-type inequality we obtain
\begin{align*}
    \sum_{k\in\N} \left\lvert f_{B_{k+1}}- f_{B_k}\right\rvert & \le  \sum_{k\in\N} \frac{1}{|B_{k+1}|}\int_{B_{k+1}}\lvert f(y)-f_{B_k}\rvert\,dy\\
    & \le  2^n \sum_{k\in\N} \frac{1}{|B_{k}|}\int_{B_{k}} \lvert f(y)-f_{B_k}\rvert\,dy\\
    & \le  c_n \sum_{k\in\N} \frac{r(B_k)}{|B_{k}|}\int_{B_{k}}|\nabla f(y)|\,dy.
\end{align*}
Let us analyze the sum:
$$\sum_{k\in\N} \frac{r(B_k)}{|B_{k}|}\int_{B_{k}}|\nabla f(y)|\,dy = \sum_{k\in\N} \frac{r(B_k)}{|B_{k}|}\int_{B_{k}\setminus B_{k+1}}|\nabla f(y)|\,dy \,+ \sum_{k\in\N} \frac{r(B_k)}{|B_{k}|}\int_{B_{k+1}}|\nabla f(y)|\, dy,
$$
but then
\begin{align*}
    \sum_{k\in\N} \frac{r(B_k)}{|B_{k}|}\int_{B_{k+1}}|\nabla f(y)|\, dy & = \sum_{k\in\N} \frac{2\,r(B_{k+1})} {2^n|B_{k+1}|}\int_{B_{k+1}}|\nabla f(y)|\, dy \\
    & =  \frac1{2^{n-1}} \sum_{k\in\N} \frac{r(B_{k+1})}{|B_{k+1}|}\int_{B_{k+1}}|\nabla f(y)|\, dy \\
    & \le \frac1{2^{n-1}} \sum_{k\in\N} \frac{r(B_{k})}{|B_{k}|}\int_{B_{k}}|\nabla f(y)|\, dy .
\end{align*}
Therefore we have
\begin{equation}\label{Estimatewithwholes}
\sum_{k\in\N}\frac{r(B_k)}{|B_k|}\int_{B_k}|\nabla f(y)|\, dy \leq c_n \sum_{k\in\N} \frac{r(B_k)}{|B_k|}\int_{B_k\setminus B_{k+1}}|\nabla f(y)|\,dy.
\end{equation}
%
Now, by the definition of $A_1$, we continue with
\begin{align*}
    \lvert f(x)-f_{B}| & \leq c_n [w]_{A_1} \sum_{k\in\N} \frac{r(B_{k})}{w(B_{k})}\int_{B_{k}\setminus B_{k+1}}|\nabla f(y)|\,w(y)\,dy \\
    & \leq c_n [w]_{A_1} \sum_{k\in\N} \frac{1}{w(B_{k})}\int_{B_{k}\setminus B_{k+1}} |x-y|\, |\nabla f(y)|\,w(y)\,dy \\
    & \leq c_n [w]_{A_1} \sum_{k\in\N} \int_{B_{k}\setminus B_{k+1}} \frac{|x-y|}{w(B(x,|x-y|))}  |\nabla f(y)|\,w(y)\,dy \\
    & \leq c_n [w]_{A_1} \int_{B} \frac{|x-y|}{w(B(x,|x-y|))}  |\nabla f(y)|\,w(y)\,dy.
\end{align*}
%
To obtain \eqref{ptwiseglobalA1}, we let $R\to \infty$ and observe that $f_B\to0$ while the right-hand side is dominated by the integral over $\R^n$.
\end{proof}

\begin{proof}[Proof of Theorem \ref{rough}]

As proved in \cite{HMP1}, we have
$$T^\star_{\Omega} f(x) \leq c_n\,\|\Omega\|_{\strt{2ex}L^{n,\infty}(\mathbb S^{n-1})} \sum_{k\in\Z}r(B_k)\avgint_{B_k}|\nabla f(y)|\,dy$$
where $B_k=B(x,2^k)$ for each $k\in\Z$. Since
$$\sum_{k\in\Z} \frac{r(B_k)}{|B_{k}|}\int_{B_{k}}|\nabla f(y)|\,\dif y = \sum_{k\in\Z} \frac{r(B_k)}{|B_{k}|}\int_{B_{k}\setminus B_{k-1}}|\nabla f(y)|\,dy \,+ \sum_{k\in\Z} \frac{r(B_k)}{|B_{k}|}\int_{B_{k-1}}|\nabla f(y)|\, dy,
$$
and
\begin{align*}
    \sum_{k\in\Z} \frac{r(B_k)}{|B_{k}|}\int_{B_{k-1}}|\nabla f(y)|\dif y & = \sum_{k\in\Z} \frac{2\,r(B_{k-1})} {2^n|B_{k-1}|}\int_{B_{k-1}}|\nabla f(y)|\, dy \\
    & =  \frac1{2^{n-1}} \sum_{k\in\Z} \frac{r(B_{k-1})}{|B_{k-1}|}\int_{B_{k-1}}|\nabla f(y)|\, dy \\
    & = \frac1{2^{n-1}} \sum_{k\in\Z} \frac{r(B_{k})}{|B_{k}|}\int_{B_{k}}|\nabla f(y)|\, dy ,
\end{align*}
we have
$$\sum_{k\in\Z}\frac{r(B_k)}{|B_k|}\int_{B_k}|\nabla f(y)|\dif y \leq c_n \sum_{k\in\Z} \frac{r(B_k)}{|B_k|}\int_{B_k\setminus B_{k-1}}|\nabla f(y)|\,dy ,
$$
and hence
$$T^\star_{\Omega} f(x) \leq c_n\,\|\Omega\|_{\strt{2ex}L^{n,\infty}(\mathbb S^{n-1})} \sum_{k\in\Z}r(B_k)\avgint_{B_k\setminus B_{k-1}}|\nabla f(y)|\,dy.$$
By the definition of $A_1$, we have
\begin{align*}
    T^\star_{\Omega} f(x) & \leq c_n \|\Omega\|_{\strt{2ex}L^{n,\infty}(\mathbb S^{n-1})}[w]_{A_1} \sum_{k\in\Z} \frac{r(B_{k})}{w(B_{k})}\int_{B_{k}\setminus B_{k-1}}|\nabla f(y)|\,w(y)\,dy \\
    & \leq c_n \|\Omega\|_{\strt{2ex}L^{n,\infty}(\mathbb S^{n-1})} [w]_{A_1} \sum_{k\in\Z} \frac{1}{w(B_{k})}\int_{B_{k}\setminus B_{k-1}} |x-y|\, |\nabla f(y)|\,w(y)\,dy \\
    & \leq c_n \|\Omega\|_{\strt{2ex}L^{n,\infty}(\mathbb S^{n-1})} [w]_{A_1} \sum_{k\in\Z} \int_{B_{k}\setminus B_{k-1}} \frac{|x-y|}{w(B(x,|x-y|))}  |\nabla f(y)|\,w(y)\,dy \\
    & \leq c_n \|\Omega\|_{\strt{2ex}L^{n,\infty}(\mathbb S^{n-1})} [w]_{A_1}\,  \int_{\R^n} \frac{|x-y|}{w(B(x,|x-y|))}  |\nabla f(y)|\,w(y)\,dy .
\end{align*}

\end{proof}

\begin{proof}[Proof of Theorem \ref{fractional}]
As proved in \cite{HMP1}, we have
$$T^\star_{\Omega} f(x) \le C \sum_{k\in \Z} \frac{1}{2^{kn}}\int_{|y|\leq 2^{k}}|\Omega(y)|\,|f(x-y)-c_k|\,dy.$$
where $c_k$ is a constant to be chosen for each $k\in\Z$. By H\"older's inequality, we have
$$T^\star_{\Omega} f(x) \leq c_n\sum_{k\in \Z} \|\Omega\|_{\strt{2ex}L^{\frac{n}{\al},\infty}(B(0,2^k))} \, \|f(x-\cdot)-c_k\|_{\strt{2ex}L^{\left(\frac{n}{\al}\right)',1}(B(0,2^k))}.$$
Since
$$|\{y\in B(0,2^k):|\Omega(y)|>\lambda\}|=\frac{2^{kn}}{n}\sigma(\{y'\in \Sn^{n-1}:|\Omega(y')|>\lambda\}),$$
we have $$\|\Omega\|_{\strt{2ex}L^{\frac{n}{\al},\infty}(B(0,2^k))}=C_n\,\|\Omega\|_{\strt{2ex}L^{\frac{n}{\al},\infty}(\Sn^{n-1})},$$
and hence
$$T^\star_{\Omega} f(x) \leq c_n\sum_{k\in \Z} \|\Omega\|_{\strt{2ex}L^{\frac{n}{\al},\infty}(\Sn^{n-1})} \, \|f(x-\cdot)-c_k\|_{\strt{2ex}L^{\left(\frac{n}{\al}\right)',1}(B(0,2^k))}.$$
Let $B_k=B(x,2^{k})$ and $c_k=f_{B_k}=\avgint_{B_k} f$, then
$$\|f(x-\cdot)-c_k\|_{\strt{2ex}L^{\left(\frac{n}{\al}\right)',1}(B(0,2^k))}=\|f-f_{B_k}\|_{\strt{2ex}L^{\left(\frac{n}{\al}\right)',1}(B_k)}.$$
This fact combined with the previous inequality yields
$$T^\star_{\Omega} f(x)\leq c_n\,\|\Omega\|_{\strt{2ex}L^{\frac{n}{\al},\infty}(\mathbb{S}^{n-1})} 
\sum_{k\in \Z}\|f-f_{B_k}\|_{\strt{2ex}L^{\left(\frac{n}{\al}\right)',1}(B_k)}.$$
We continue by using the Poincar\'e-Sobolev inequality \eqref{PS-BBM-Lorentz}. 
 Recalling the definition of $\mathcal D^{\al}$
$$\mathcal D^{\al}f(x)=\int_{\R^n}\frac{|f(x)-f(y)|}{|x-y|^{n+\al}}\,dy,$$
we can continue with, 

\begin{align*}
    T^\star_{\Omega} f(x) & \leq c_n(1-\al)\,\|\Omega\|_{\strt{2ex}L^{\frac{n}{\al},\infty}(\mathbb{S}^{n-1})}\sum_{k\in \Z} r(B_k)^{\al} \avgint_{B_k} \int_{B_k} \frac{\lvert f(z)-f(y)\rvert}{\lvert z-y\rvert^{n+\al}}\,dy\,dz \\
    & \leq c_n(1-\al)\,\|\Omega\|_{\strt{2ex}L^{\frac{n}{\al},\infty}(\mathbb{S}^{n-1})}\sum_{k\in \Z}
r(B_k)^{\al} \avgint_{B_k} \mathcal D^{\al}f(y) \,dy,
\end{align*}
We now analyze the sum:
\begin{align*}
    \sum_{k\in \Z} r(B_k)^{\al} \avgint_{B_k} \mathcal D^{\al}f(y) \,dy & = \sum_{k\in \Z} \frac{2^{k\al}}{\omega_n2^{kn}} \int_{B_k} \mathcal D^{\al}f(y) \,dy \\
    & =\frac1{\omega_n}\sum_{k\in \Z}\frac{1}{2^{k(n-\al)}}\int_{2^{k-1}<|x-y|\leq2^k} \mathcal D^{\al}f(y) \,dy \\
    & \qquad\qquad\qquad +\frac{1}{\omega_n}\sum_{k\in \Z}\frac{1}{2^{k(n-\al)}}\int_{|x-y|\leq 2^{k-1}} \mathcal{D}^{\al}f(y) \,dy \\
    & \leq \frac1{\omega_n}\sum_{k\in \Z}\int_{2^{k-1}<|x-y|\leq2^k}\frac{1}{|x-y|^{n-\al}}\, \mathcal D^{\al}f(y) \,dy \\
    & \qquad\qquad\qquad + \frac1{2^{n-\al}}\,\sum_{k\in \Z}r(B_{k-1})^{\al} \avgint_{B_{k-1}} \mathcal D^{\al}f(y) \,dy \\
    & \leq  \frac1{\omega_n}I_{\al}(\mathcal D^{\al}f)(x)+\frac1{2^{n-\al}}\,\sum_{k\in \Z} r(B_k)^\al\avgint_{B_k}\mathcal D^{\al}(y) \,dy.
\end{align*}
Since $\al<1$ and $n\geq2$, we have $\frac1{2^{n-\al}}<1$, and we can rearrange the terms to obtain
%
%
\begin{align*}
\sum_{k\in \Z} r(B_k)^{\al} \avgint_{B_k} \mathcal D^{\al}f(y)\, dy& \leq c_n(2^{n-\al})'\,I_{\al}(\mathcal D^{\al}f)(x)\\
&  \leq c_n\,I_{\al}(\mathcal D^{\al}f)(x).
\end{align*}
This concludes the proof of the theorem.

\end{proof}

\begin{proof}[Proof of Lemma \ref{I1-domination}]
    Since
$$f(y)-f(x) = \int_0^1 \nabla f(x+t(y-x)) \cdot (y-x) \,dt$$
by the Fundamental Theorem of Calculus for line integrals, we have
\begin{align*}
    I_{\al}(\mathcal D^{\al}f)(x) & = \int_{\R^n} \int_{\R^n}\frac{|f(y)-f(z)|}{|y-z|^{n+\al}}\,dz \frac{dy}{|x-y|^{n-\al}} \\
    & \le \int_{\R^n} \int_{\R^n} \int_0^1 |\nabla f(y+t(z-y))| \,dt \frac{dz}{|y-z|^{n+\al-1}}\, \frac{dy}{|x-y|^{n-\al}} \\
    & = \int_{\R^n} \int_0^1 \int_{\R^n} \frac{|\nabla f(y+t(z-y))|}{|y-z|^{n+\al-1}} dz \,dt\, \frac{dy}{|x-y|^{n-\al}} \\
    & = \int_{\R^n} \left(\int_{\R^n} \frac{|\nabla f(u)|}{|u-y|^{n+\al-1}} du \int_0^1\frac{t^{n+\al-1}}{t^n}\,dt\right) \frac{dy}{|x-y|^{n-\al}} \\
    & = \frac{1}{\al} \int_{\R^n} \int_{\R^n} \frac{|\nabla f(u)|}{|u-y|^{n+\al-1}} du \, \frac{dy}{|x-y|^{n-\al}} \\
    & = \frac{1}{\al} \int_{\R^n} |\nabla f(u)| \int_{\R^n} \frac{1}{|u-y|^{n+\al-1}} \frac{1}{|x-y|^{n-\al}} \,dy\,du
\end{align*}
where we have made the change of variables $u = y+t(z-y)$ in the fourth estimate. Now, we use the beta integral identity (see exercise 2.4.9 in  \cite{G}), namely
$$\int_{\R^n} \frac{ 1}{|t-x_1|^{\al_1}}\, \frac{ 1}{|t-x_2|^{\al_2}}\,dt= \pi^{n/2}
\frac{\Gamma(\frac{n-\alpha_1}{2})}{\Gamma(\frac{\alpha_1}{2})}
\frac{\Gamma(\frac{n-\alpha_2}{2})}{\Gamma(\frac{\alpha_2}{2})}
\frac{\Gamma(\frac{\alpha_1+\alpha_2-n}{2})}{\Gamma(n-\frac{\alpha_1+\alpha_2}{2})}
|x_1-x_2|^{n-\al_1-\al_2}
$$
for $0<\alpha_1,\alpha_2<n$, $\alpha_1+\alpha_2>n$. 

Applying this with $\alpha_1=n+\al-1$ and $\alpha_2=n-\al$, we have
\begin{align*}
    I_{\al}(\mathcal D^{\al}f)(x) & \leq \frac{\pi^{n/2}\,\Gamma(\frac{1-\al}{2})\,\Gamma(\frac{\al}{2})\,\Gamma(\frac{n-1}{2})}{\al\,\Gamma(\frac{n+\al-1}{2})\,\Gamma(\frac{n-\al}{2})\,\Gamma(\frac{1}{2})} \int_{\R^n} \frac{ |\nabla f(u)|}{|x-u|^{n-1}}\,du \\
    & = \frac{\pi^{\frac{n-1}{2}}\,\Gamma(\frac{1-\al}{2})\,\Gamma(\frac{\al}{2})\,\Gamma(\frac{n-1}{2})}{\al\,\Gamma(\frac{n+\al-1}{2})\,\Gamma(\frac{n-\al}{2})} \, I_1(|\nabla f|)(x).
\end{align*}
\end{proof}

\begin{proof}[Proof of Theorem \ref{identityfractional}]
As in the proof of Theorem \ref{identity}, we have
$$\left\lvert f(x)- f_B\right\rvert \le  c_n \sum_{k\in\N} \frac{1}{|B_{k}|}\int_{B_{k}} \lvert f(y)-f_{B_k}\rvert\,dy,$$
 where $B$ is any ball in $\R^n$ centered at $x$. 
Applying \eqref{PS-BBMspecial} yields
\begin{align*}
    \left\lvert f(x)- f_B\right\rvert & \le c_n (1-\al)\sum_{k\in\N}r(B_k)^{\al} \avgint_{B_k} \int_{B_k}\frac{|f(z)-f(y)|}{|z-y|^{n+\al  }}\,dy\,dz \\
    & =  c_n (1-\al)\sum_{k\in\N}r(B_k)^{\al} \avgint_{B_k} \mathcal{D}^\al f(y)\,dy.
\end{align*}
Since
$$\sum_{k\in\N}r(B_k)^{\al} \avgint_{B_k} \mathcal{D}^\al f(y)\,dy \le \sum_{k\in\N}\frac{r(B_k)^{\al}}{|B_k|} \int_{B_k\setminus B_{k+1}} \mathcal{D}^\al f(y)\,dy \,+ \sum_{k\in\N}\frac{r(B_k)^{\al}}{|B_k|} \int_{B_{k+1}} \mathcal{D}^\al f(y)\,dy,$$
and
\begin{align*}
    \sum_{k\in\N}\frac{r(B_k)^{\al}}{|B_k|} \int_{B_{k+1}} \mathcal{D}^\al f(y)\,dy & = \sum_{k\in\N}\frac{2^\al r(B_{k+1})^{\al}}{2^n|B_{k+1}|} \int_{B_{k+1}} \mathcal{D}^\al f(y)\,dy \\
    & \le 2^{\al-n}\sum_{k\in\N}\frac{r(B_k)^{\al}}{|B_k|} \int_{B_k} \mathcal{D}^\al f(y)\,dy,
\end{align*}
we have
$$\left\lvert f(x)- f_B\right\rvert \le c_n (1-\al)(2^{n-\al})'\sum_{k\in\N}\frac{r(B_k)^{\al}}{|B_k|} \int_{B_k\setminus B_{k+1}} \mathcal{D}^\al f(y)\,dy.
$$
Since $w\in A_1$, we have
\begin{align*}
    \left\lvert f(x)- f_B\right\rvert & \le c_n (1-\al)(2^{n-\al})'\,[w]_{A_1}\sum_{k\in\N}\frac{r(B_k)^{\al}}{w(B_k)} \int_{B_k\setminus B_{k+1}} \mathcal{D}^\al f(y)\,w(y)\,dy \\
    & \le c_n (1-\al)2^\al(2^{n-\al})'\, [w]_{A_1}\sum_{k\in\N}\int_{B_k\setminus B_{k+1}} \frac{|x-y|^\al\,\mathcal{D}^\al f(y)}{w(B(x,|x-y|))}\,w(y)\,dy \\
    & = c_n (1-\al)2^\al(2^{n-\al})'\, [w]_{A_1}\int_{B} \frac{|x-y|^\al}{w(B(x,|x-y|))}\,\mathcal{D}^\al f(y)\,w(y)\,dy.
\end{align*}
As before, the bound $2^\al(2^{n-\al})'$ is bounded by a dimensional constant since $\al<1$ and $n\geq 2$. 
Finally, by letting $R\to\infty$, we obtain \eqref{globalfrac} concluding the proof of the theorem.  
\end{proof}

\hfill

\begin{proof}[Proof of Theorem \ref{roughfractional}]
As in the proof of Theorem \ref{fractional}, we have
\begin{align*}
    T^\star_{\Omega} f(x) & \leq  c_n(1-\al)\,\|\Omega\|_{\strt{2ex}L^{\frac{n}{\al},\infty}(\mathbb{S}^{n-1})}\sum_{k\in \Z}\frac{1}{2^{k(n-\al)}}\int_{B_k\setminus B_{k-1}} \mathcal D^{\al}f(y) \,dy
\end{align*}
where $B_k=B(x,2^k)$.
Since $w\in A_1$, we have
\begin{align*}
    T^\star_{\Omega} f(x) & \leq C_{n,\al,\Omega}\, \sum_{k\in\Z} \frac{r(B_{k})^\al}{w(B_{k})}\int_{B_{k}\setminus B_{k-1}}\mathcal D^{\al}f(y)\,w(y)\,dy \\
    & \leq  C_{n,\al,\Omega}\, \sum_{k\in\Z} \frac{1}{w(B_{k})}\int_{B_{k}\setminus B_{k-1}} |x-y|^\al\, D^{\al}f(y)\,w(y)\,dy \\
    & \leq  C_{n,\al,\Omega}\, \sum_{k\in\Z} \int_{B_{k}\setminus B_{k-1}} \frac{|x-y|^\al}{w(B(x,|x-y|))} D^{\al}f(y)\,w(y)\,dy \\
    & \leq  C_{n,\al,\Omega}\, \int_{\R^n} \frac{|x-y|^\al}{w(B(x,|x-y|))} D^{\al}f(y)\,w(y)\,dy
\end{align*}
where $C_{n,\al,\Omega}=c_n(1-\al)\,\|\Omega\|_{\strt{2ex}L^{\frac{n}{\al},\infty}(\mathbb{S}^{n-1})}$.
\end{proof}

\hfill

\section{Consequences of our main results}

Recall that a weight $w$ is said to be doubling if there exists a constant $D=D(w)$ such that
$$w(B(x,2r))\leq D\,w(B(x,r))$$
for all $x\in\R^n$ and all $r>0$.

%



\begin{theorem}\label{TwLpLq} Let $d$ be a positive constant bigger than one. 
Let $p,q$ be numbers such that $1<p<d$ and $\frac{1}{p}-\frac{1}{q}=\frac{1}{d}$. Let $w$ be a  doubling weight and assume further that there exists a positive constant $c=c(w)$ such that
\begin{equation}\label{lowerAlhfors}
    w(B(x,r))\geq c\,r^d, \qquad  x\in\R^n, r>0, 
\end{equation}
then
$$\|T_wf\|_{L^q(w)}\leq C_{w,p,d}\,\|f\|_{L^p(w)}.$$
\end{theorem}

\begin{remark}
We want to emphasize the need to assume condition \eqref{lowerAlhfors}. Although it is known that the doubling condition implies the lower Alhfors condition \eqref{lowerAlhfors}, this is only true for bounded spaces $X$, i.e. $\textsf{diam}(X)<\infty$. When $X$ is unbounded, the doubling condition no longer implies the lower Alhfors condition. For simplicity, consider $n=1$, $X=\R$ and $w(x)=|x|^{-\frac{1}{2}}$. Then
$$w(B(x,r))=\int_{x-r}^{x+r}|y|^{-\frac{1}{2}}\,dy.$$
This weight $w$ is an $A_1$ weight and hence doubling (\cite{G} (example 9.1.6)). However, $w$ does not satisfy the lower Alhfors condition. Indeed, suppose that there are constants $c,d>0$ such that
$$\int_{x-r}^{x+r}|y|^{-\frac{1}{2}}\,dy\geq c\,r^d$$
for all $x\in\R$ and all $r>0$. Fix $r>0$ and consider $x>r$, then we have
$$2(\sqrt{x+r}-\sqrt{x-r})\geq c\,r^d\iff \frac{4r}{\sqrt{x+r}+\sqrt{x-r}}\geq c\,r^d,$$
which is a contradiction as $x\to\infty$.
\end{remark}

\begin{corollary}
Suppose $1<p<n$ and $\frac{1}{p}-\frac{1}{p^*}=\frac{1}{n}$. Let $w$ be a doubling weight and $c=c(w)$ be a positive constant such that
\begin{equation}
    w(x)\geq c,   \qquad  a.e. x\in\R^n.
\end{equation}
Then
$$\|T_wf\|_{L^{p^*}(w)}\leq C_{w,p}\,\|f\|_{L^p(w)}.$$

\end{corollary}

Indeed, observe that \eqref{lowerAlhfors} becomes $w(B(x,r))\geq c\,r^n$ whenever $d=n$, and hence by the Lebesgue Differentiation Theorem we have $w(x) \geq c$, almost everywhere. The conditions are equivalent in this case.

From \cite{G} (Proposition 9.1.5), if $w\in A_1$, then $w$ is doubling; more precisely, we have
$$w(B(x,\lambda r))\leq \lambda^n[w]_{A_1}w(B(x,r))$$
for all $x\in\R^n$, all $r>0$ and all $\lambda>0$. Using this, we have the following corollary.

\begin{corollary}
Suppose $1<p<n$ and $\frac{1}{p}-\frac{1}{p^*}=\frac{1}{n}$. Let $w$ be an $A_1$ weight such that for a positive constant $c=c(w)$
\begin{equation}
    w(x)\geq c,   \qquad  a.e.\ x\in\R^n.
\end{equation}
Then
$$\|f\|_{L^{p^*}(w)}\leq C_{w,p,d}\,\|\nabla f\|_{L^p(w)}$$
for all $f\in C_c^\infty(\R^n)$.
\end{corollary}

\begin{proof}[Proof of Theorem \ref{TwLpLq}] 
For the proof we make a variation of the well known method of Hedberg  \cite{H}.
For any $R>0$, we have $T_wf(x)=T_w^1f(x)+T_w^2f(x)$ where
$$T_w^1f(x)=\int_{|y-x|<R}\frac{|y-x|}{w(B(x,|y-x|))}\,f(y)\,w(y)\,dy$$
and
$$T_w^2f(x)=\int_{|y-x|\geq R}\frac{|y-x|}{w(B(x,|y-x|))}\,f(y)\,w(y)\,dy.$$
We first estimate $T^1_wf(x)$ as follows:
\begin{align*}
    T_w^1f(x) & =\sum_{k=1}^\infty\int_{\frac{R}{2^k}<|y-x|\leq\frac{R}{2^{k-1}}} \frac{|y-x|}{w(B(x,|y-x|))}\,f(y)\,w(y)\,dy \\
    & \leq \sum_{k=1}^\infty \frac{R}{2^{k-1}}\cdot\frac{1}{w(B(x,\frac{R}{2^k}))} \int_{|y-x|\leq\frac{R}{2^{k-1}}} f(y)\,w(y)\,dy \\
    & \leq \sum_{k=1}^\infty \frac{R}{2^{k-1}}\cdot\frac{C_w}{w(B(x,\frac{R}{2^{k-1}}))} \int_{B(x,\frac{R}{2^{k-1}})} f(y)\,w(y)\,dy \\
    &\leq C_w\,R\,M^c_wf(x)
\end{align*}
where we have utilized the doubling property of $w$ in the third estimate.
For $T^2_wf(x)$, by H\"older's inequality, we have:
\begin{align*}
    T_w^2f(x) & \leq \|f\|_{L^p(w)}\left(\int_{|y-x|\geq R} \frac{|y-x|^{p'}}{w(B(x,|y-x|))^{p'}}\,w(y)\,dy\right)^\frac{1}{p'} \\
    & = \|f\|_{L^p(w)}\left(\sum_{k=1}^\infty\int_{2^{k-1}R<|y-x|\leq2^kR} \frac{|y-x|^{p'}}{w(B(x,|y-x|))^{p'}}\,w(y)\,dy\right)^\frac{1}{p'} \\
    & \leq R\,\|f\|_{L^p(w)}\left(\sum_{k=1}^\infty\frac{2^{kp'}w(B(x,2^kR))}{w(B(x,2^{k-1}R))^{p'}}\right)^\frac{1}{p'} \\
    & \leq C_{w,p}\,R\,\|f\|_{L^p(w)}\left(\sum_{k=1}^\infty\frac{2^{kp'}}{w(B(x,2^{k-1}R))^{p'-1}}\right)^\frac{1}{p'} \\
    &\leq C_{w,p}\,R^{1-\frac{d}{p}} \|f\|_{L^p(w)}\left(\sum_{k=1}^\infty\frac{2^{kp'}}{2^{(k-1)(p'-1)d}}\right)^\frac{1}{p'}
\end{align*}
where we have utilized the doubling condition on $w$ in the fourth estimate and condition \eqref{lowerAlhfors} in the last estimate. Since $(p'-1)d=\frac{d}{p-1}>\frac{p}{p-1}=p'$, the sum in the last display converges, and hence
$$ T_w^2f(x) \leq C_{w,p,d}\,R^{1-\frac{d}{p'}} \|f\|_{L^p(w)}.$$
This estimate and the estimate for $T^1_wf(x)$ imply
$$T_wf(x) \leq C_{w,p,d}\left[R^{1-\frac{d}{p}} \|f\|_{L^p(w)}+R\,M^c_wf(x)\right].$$
Notice that the right-hand side attains its minimum value at
$$R=\left(\frac{(\frac{d}{p}-1)\|f\|_{L^p(w)}}{M^c_wf(x)}\right)^\frac{p}{d},$$
so we have
$$T_wf(x) \leq C_{w,p,d}\, \|f\|_{L^p(w)}^\frac{p}{d}\,M^c_wf(x)^{1-\frac{p}{d}}.$$
Recall that the centered weighted maximal function is defined by
$$M^c_wf(x)\coloneq\sup_{r>0}\frac{1}{w(B(x,r))}\int_{B(x,r)}f(y)\,w(y)\,dy$$
is bounded from $L^p(w)$ to $L^p(w)$. Therefore, we have
\begin{align*}
    \|T_wf\|_{L^q(w)} & \leq C_{w,p,d}\, \|f\|_{L^p(w)}^\frac{p}{d}\,\|(M^c_wf)^{1-\frac{p}{d}}\|_{L^q(w)} \\
    & = C_{w,p,d}\, \|f\|_{L^p(w)}^\frac{p}{d}\,\|M^c_wf\|_{L^p(w)}^{1-\frac{p}{d}} \\
    & \leq C_{w,p,d}\, \|f\|_{L^p(w)}^\frac{p}{d}\,\|f\|_{L^p(w)}^{1-\frac{p}{d}} = \|f\|_{L^p(w)}.
\end{align*}
\end{proof}

\end{document}